\documentclass[10 pt]{article}
\usepackage{amssymb}
\usepackage{amsthm}
\usepackage{amsmath}
\usepackage{mathrsfs}
\usepackage{verbatim}
\begin{document}
\title{Remarks on the semi-classical Hohenberg-Kohn functional\footnote{The author is pleased to acknowledge the support of a University of Alberta start-up grant and a National Sciences and Engineering Research Council of Canada Discovery Grant.}}\author{Brendan Pass\footnote{Department of Mathematical and Statistical Sciences, 632 CAB, University of Alberta, Edmonton, Alberta, Canada, T6G 2G1 pass@ualberta.ca.}}
\maketitle

\begin{abstract}
In this note, we study an optimal transportation problem arising in density functional theory.  We derive an upper bound on the semi-classical Hohenberg-Kohn functional derived by Cotar, Friesecke and Kl\"{u}ppelberg \cite{cfk2} which can be computed in a straightforward way for a given single particle density.  This complements a lower bound derived by the aforementioned authors.   We also show that for radially symmetric densities the optimal transportation problem arising in the semi-classical Hohenberg-Kohn functional can be reduced to a $1$-dimensional problem.  This yields a simple new proof of the explicit solution to the optimal transport problem for two particles found in \cite{cfk}.  For more particles, we use our result to demonstrate two new qualitative facts: first, that the solution can concentrate on \textit{higher dimensional} submanifolds and second that the solution can be \textit{non-unique}, even with an additional symmetry constraint imposed.
\end{abstract}

\section{Introduction}
In this paper, we study a multi-marginal optimal transportation problem arising in density functional theory 	(DFT) in condensed matter physics.  Optimal transportation is the general problem of coupling two (or, in our case, $N$) probability measures together as efficiently as possible, relative to a given cost function $c$.  This is a rapidly expanding area of mathematical research, with many diverse applications; recent progress is described in the books by Villani \cite{V,V2}.  

To precisely formulate our problem, fix a probability measure $\rho$ on $\mathbb{R}^d$ (the most physically relevant case being $d=3$).  Let $\Pi(\rho)$ be the set of all probability measures $\rho_N$ on $\mathbb{R}^{dN}$ whose marginals are all $\rho$; that is, $(\pi_{i})_{\#} \rho_N = \rho$ for $i=1,2,...,N$, where $\pi_i:\mathbb{R}^{dN} \rightarrow \mathbb{R}^d$ is the $i$th canonical projection.  We then define

\begin{equation*}
C_N(\rho_N):=\int_{\mathbb{R}^{dN}}\sum_{i \neq j} \frac{1}{|x_i-x_j|} d\rho_{N}
\end{equation*}
and
\begin{equation}\label{mk}
E_N[\rho]:=\inf_{\rho_N \in \Pi(\rho)}C_N(\rho_N)
\end{equation}

Readers familiar with optimal transportation will recognize this as the \emph{multi-marginal Monge-Kantorovich optimal transportation problem,} with equal marginals $\rho$ and cost function $c(x_1,x_2,...,x_N):=\sum_{i\neq j}\frac{1}{|x_i-x_j|}$, which we will refer to as the Coulomb cost.  

Density functional theory is a modeling method used by physicists and chemists to understand electron correlations.  It was originally proposed by Hohenberg, Kohn and Sham \cite{HK}\cite{KohnSham}; see \cite{PY} and \cite{FNM} for a detailed introduction. The \emph{Hohenberg-Kohn functional} plays a central role here.  Imagine a system of $N$ electrons, interacting via the Coulomb potential.  Given a prescribed single particle density $\rho$, this functional returns the minimum energy among all $n$-particle wave functions $\psi$ whose single particle density is $\rho$.    The Hohenberg-Kohn functional is given by:

\begin{equation*} 
F_{HK}[\rho]:=\inf_{\psi \rightarrow \rho}\Big( \frac{\hbar}{2m_e}\int_{\mathbb{R}^{dN}}\sum_{i=1}^{N} |\nabla_{x_i}\psi|^2 d\rho_N(x_1,x_2,...,x_N) + {N \choose 2}\int_{\mathbb{R}^{2d}}\frac{1}{|x_1-x_2|}d\rho_2(x_1,x_2)\Big).
\end{equation*}

Here, $\hbar$ is Planck's constant over $2\pi$ and $m_e$ is the mass of an electron.  The notation $\psi \rightarrow \rho$ means that $\rho$ is the single particle density corresponding to the wave function $\psi$, and $\rho_N$ and $\rho_2$ are the $N$ and two particle densities corresponding to $\psi$, respectively. The first term is the quantum mechanical kinetic energy while the second is the Coulomb interaction energy.   Numerically, this expression is unwieldy for large $N$, as the complexity of minimizing over the space of $N$ particle wave-function grows exponentially in $N$.  It turns out that the kinetic energy term can be dealt with relatively easily (see \cite{cfk} for details); one of the major goals of DFT is to approximate the interaction energy of two electrons by a function of the single particle density $\rho$.

Two recent papers by Cotar, Friesecke and Kl\"{u}ppelberg \cite{cfk}\cite{cfk2}, as well as a paper by Buttazzo, De Pascale and Gori-Giorgi \cite{bdpgg}, have revealed interesting connections between this problem and optimal transportation.\footnote{In fact, these connections were already implicitly present in the physics literature, but without rigorous justification \cite{sggs}\cite{seidl}.}  Of particular interest to us in the present work, the work in \cite{cfk2} showed that in the semi-classical limit, $\hbar \rightarrow 0$, the Hohenberg-Kohn functional reduces to $E_N$.  In particular, contributions from the kinetic energy vanish and the antisymmetry property associated with $N$-body wave functions reduces to \textit{symmetry} of the measure $\rho_N$ in the arguments $(x_1,x_2,...,x_N)$.  In fact, given any measure $\rho_N \in \Pi(\rho)$, we have $C_N(\rho_N) = C_N(\tilde{\rho_N})$, where $\tilde{\rho_N}$ is the symmetrization of $\rho_N$; that is, for all Borel $A \subseteq \mathbb{R}^{Nd}$

\begin{equation*}
\tilde{\rho_N}(A) = \frac{1}{N!}\sum_{\sigma \in S_N} \rho_N \big(\{(x_{\sigma(1)},x_{\sigma(2)},...x_{\sigma(N)}): (x_1,x_2,...,x_N) \in A\}\big),
\end{equation*}
where $S_N$ is the permutation group on $N$ symbols.  Therefore, when computing $E_N[\rho]$, we can neglect the symmetry condition entirely.

Optimal transportation problems with two marginals (ie, $N=2$) with a variety of different cost functions have been studied extensively, including results in \cite{cfk} and \cite{bdpgg}on the Coulomb cost; see, for example \cite{bren} \cite{gm} \cite{G} \cite{Caf}, or, for a detailed overview and extensive bibliography, see \cite{V} and \cite{V2}.   However, relatively little is known about optimal transportation problems with more than two marginals, except in certain special cases \cite{gs}\cite{CN}\cite{KS}\cite{C}\cite{P}\cite{P1}\cite{RU}\cite{RU2}\cite{OR}.  

In particular, despite its importance in density functional theory, to this point very little is known about the structure of minimizers of $C_N$ for $N>2$, or the minimal values $E_N[\rho]$.  The main goal of this note is to contribute to this problem. 
  
This paper features two main contributions.  The first is the construction of a measure $\rho_N \in \pi(\rho)$; the immediately yields an upper bound on $E_N[\rho]$, which can be calculated explicitly from the density $\rho$.  When $d=1$, we suspect that this measure is optimal and prove this in a special case, although we leave the general case open.  In higher dimensions, our measure couples every line through the origin to itself, using the aforementioned $1$-dimensional construction.  A lower bound on $E_N[\rho]$ was identified  in \cite{cfk}; precisely,  $E_N[\rho] \geq {N \choose 2}E_2[\rho]=\frac{N(N-1)}{2} E_2[\rho]$.  \footnote{In general, this inequality may be strict, as the optimal measure $\rho_2$ for the two marginal problem may not be the two particle density of some $N$-particle wave function.  As was pointed out in \cite{cfk}, this has to do with the representibility problem for two particle density matrices.  In the semi-classical limit, where wave functions are essentially replaced by probability measures, this is the question of the existences of a measure on $\mathbb{R}^{Nd}$ whose $2d$-dimensional marginals are all equal to $\rho_2$.}  For radially symmetric measures when $N=2$, Cotar, Friesecke and Kl\"{u}ppelberg were able to construct the optimal measure explicitly and therefore calculate $E_2[\rho]$.  Therefore, in this case, we have a lower bound which can be calculated and therefore can be compared to the upper bound we derive.  For one particularly simple $\rho$ we do this.  We note that when $N=2$ and $\rho$ is radially symmetric, our construction coincides with the explicit solution in \cite{cfk}.  For larger $N$, it is not generally the optimal measure.

This may be useful to physicists and chemists, as the explicit construction can be used as a starting point for constructing approximations to the Hohenberg-Kohn functional.  In fact, as is emphasized in \cite{cfk}, the fundamental goal of DFT is to approximate the electron interaction energy (which depends on the two particle density) by a functional of the single particle density.  Our upper bound, which can be explicitly calculated from the single particle density $\rho$, could be a good starting point for more sophisticated approximations.

Our second contribution concerns radially symmetric measures.  In this case, we show that the problem can be reduced to a $1$-dimensional optimal transport problem with a certain effective cost function derived from the Coulomb cost (but not equal to it).  This has two immediate applications.  First, we use this result to provide a  new proof of the optimality of a construction in \cite{cfk} and \cite{bdpgg} for two marginals, which is considerably simpler than the original proof.  Secondly, we show that this reduction implies that for more marginals the support of the optimizer can be more than $d$-dimensional and be non-unique (even with the additional symmetry constraint. This stands in stark contrast to the two marginal case, where solutions are known to be unique and concentrated on the graphs of functions over $x_1$ and so are essentially $d$-dimensional. Similar phenomena, have, however, been observed in multi-marginal problems for other cost functions \cite{CN}\cite{P}.  Let us note that non-uniqueness of minimizers of $C_N$ is also ready present in the $d=1$ case when $N>2$ (see Remark \ref{nonunique}) and is perhaps not too surprising.  On the other hand, I found the fact that there can be more than one minimizer which is \textit{symmetric in the arguments $(x_1,x_2,...,x_N)$} surprising; from a physical standpoint symmetric measures $\rho_N$ represent the semi-classical limits of single particle densities arising from wave functions.

Aside from these direct applications, this reduction result should be useful in deriving numerical methods to evaluate the Hohenberg-Kohn functional, as it reduces an optimization problem over measures on $\mathbb{R}^{Nd}$ to a problem over measures on $\mathbb{R}^{N}$

The idea that the optimizer should come from aligning the radial densities optimally, and positioning the particles on fixed spheres to minimize the Coulomb interaction is implicit in the physics literature \cite{sggs}, but has not been made rigorous until now.  In particular, the explicit construction in the $N=2$ case, though present in \cite{sggs} and \cite{seidl}, was not proven to be optimal until the work in \cite{cfk} and \cite{bdpgg}, by a fairly complicated proof, relying on a general theorem guaranteeing the existence and uniqueness of Monge type solutions.  Our proof here, essentially making rigorous the intuition in \cite{sggs}, is much simpler.  On the other hand, the non-uniqueness we find for larger $N$ does not seem to have been anticipated at all.

The next section is devoted to the derivation of our upper bound on $E_N[\rho]$.  In section 3, we restrict our attention to the radially symmetric case and show that the problem can be reduced to an optimal transportation problem with $1$-dimensional marginals.

I would like to thank Robert McCann for originally pointing out the work in \cite{cfk} to me.  I am also grateful to both Robert and Codina Cotar for very interesting and useful discussions on this topic, and to Gero Friesecke for a very insightful explanation of the form of the semi-classical limit of the Hohenberg-Kohn functional.

\section{An upper bound on $E_N[\rho]$.}
Here we construct a measure $\rho_N \in \Pi(\rho)$, yielding an upper bound on $E_N[\rho]$.  We do this by coupling every line through the origin to itself, via a $1$-dimensional coupling described below.
\subsection{The $1$-dimensional problem}
In this subsection, we consider (\ref{mk}) when $d=1$.  We construct a measure $\rho_N \in \Pi(\rho)$ which we suspect, but do not prove, attains the minimum in (\ref{mk}).  Given a probability measure $\rho$ on $\mathbb{R}$, we divide the real line $\mathbb{R}$ into $N$ subintervals, each with equal mass.  Our measure will then couple together mass from distinct intervals in each of the $N$ copies of $\mathbb{R}$.  This construction ensures that all of the electrons remain isolated from each other and generalizes the $1$-dimensional construction in \cite{cfk} in the radially symmetric case when $N=2$.  This construction is essentially the same as the ones in \cite{bdpgg} and \cite{seidl}, but it's use in constructing an upper bound for non-radial densities is, to the best of my knowledge, new.

Set $r^0 = -\infty$.  Now define $r^i$ recursively by 
\begin{equation*}
r^i=\inf\{t^i:\rho(r^{i-1},t^i) \leq \frac{1}{N}\}, 
\end{equation*}
for $i=1,2,3,...,N-1$.  Set $r^N =\infty$.  Assuming $\rho$ is absolutely continuous with respect to Lebesgue measure, this yields $N$ subintervals of $\mathbb{R}$, $I^i=[r^{i-1},r^i]$, each with $\rho(I^i) = \frac{1}{N}$.  Define $F^i: I^1 \rightarrow I^i$ implicitly by 
\begin{equation*}
\rho(r^{i-1},F^i(t))= \rho(-\infty, t).
\end{equation*}
  Then, letting $\rho^i$ be $\rho$ restricted to $I^i$, we have $F^i_{\#}\rho^1 = \rho^i$ by construction.  It is also worth noting that each $F^i$ is an increasing function, and that $F^1$ is the identity function, $F^1(t)=t$.

Now, we can extend $F^i: \mathbb{R} \rightarrow \mathbb{R}$ as follows: take the image $F^i(I^j)$ to be $I^{j+i-1}$ (here addition is modulo $N$, so that, for example, $I^{N+1} = I^1$).  $F^i : I^j \rightarrow I^{j+i-1}$ is defined implicitly by 
\begin{equation*}
\rho(r^{i+j-2},F^i(t))= \rho(r^{j-1}, t).  
\end{equation*}
It it then clear that $(F^i)_{\#} \rho = \rho$.  Setting $F = (F^1,F^2,....F^N)$, the measure $\rho_N:=F_{\#} \rho$ is in $\Pi(\rho)$.

\newtheorem{local}{Remark}[subsection]

\begin{local}
(Local optimality of $\rho_N$) Note that the measure $\rho_N$ is concentrated on the union of sets of the form $P^i=I^i \times I^{i+1} \times ...\times I^{i+N-1}$, for $i =1,2,...N$ (here, as above, the indices $j$ in $I^j$ should be understood modulo $N$).  Note that $\frac{\partial ^2 c}{\partial x^i \partial x^j} < 0$ for all $i \neq j$.  Now, consider the optimal transportation problem on $P^i$, with marginals $\rho^{i}, \rho^{i+1},...,\rho^{i+N-1}$.  A theorem of Carlier \cite{C}, rediscovered in \cite{P4}, demonstrates that the measure $\rho_N|_{P^i}=F_{\#}\rho_i$ above is the optimal measure.
\end{local}

We suspect that $\rho_N$ is optimal, but do not prove this in general.  Below, we prove this is the special case where $\rho$ is uniform measure on $[0,1]$; a similar argument for the $N=3$ case was worked out in \cite{bdpgg}. 
\newtheorem{1d}[local]{Theorem}
\begin{1d}
Suppose $\rho$ is uniform measure on $[0,1]$.  Then $\rho_N$ is optimal.
\end{1d}

\begin{proof}
As $\rho$ is uniform measure on $[0,1]$, the interval $I^i = [(i-1)/N,i/N]$, and the map $F^i(x_j) = x_j +(i+j-1)/N$, for $x_j \in I^j$.  

Define a continuous function $u:[0,1] \rightarrow \mathbb{R}$ as follows.  Set $u(0) =0$.  For $x \in I^i$, set

\begin{equation*}
\frac{d u}{d x}(x) = \sum_{j < i} \frac{-N^2}{(i-j)^2} + \sum_{j > i} \frac{N^2}{(i-j)^2}
\end{equation*}
Note that this implies  $\frac{d u}{d x}(x)$ is constant on $I^i$, and that $\frac{d u}{d x}(x)$ is non-increasing, so that $u$ is piecewise linear and concave.

We will now show that $c(x_1,x_2,...,x_N) - \sum_{i=1}^Nu(x_i)$ attains its minimum value at every point in the support $spt(\rho_N)$ of $\rho_N$.  This will imply the desired result, as for any measure $\gamma \in \Pi(\rho)$ we have 

\begin{eqnarray*}
\int_{\mathbb{R}^{Nd}} c(x_1,x_2,...,x_N) d\gamma & = & \int_{\mathbb{R}^{Nd}} \big[c(x_1,x_2,...,x_N) - \sum_{i=1}^Nu(x_i) + \sum_{i=1}^Nu(x_i)\big] d\gamma\\
&\geq& M  +  \sum_{i=1}^N \int_{\mathbb{R}^{d}} u(x)) d\rho(x) 
\end{eqnarray*}
with equality when $\gamma = \rho_N$ (here $M$ is the minimum value of $c(x_1,x_2,...,x_N) - \sum_{i=1}^Nu(x_i)$).

Consider the set 

\begin{equation*}
S = \bigcup_{\sigma} \{(F^{\sigma(1)}(x), F^{\sigma(2)}(x),...,F^{\sigma(N)}(x))| x \in [0,1]\}, 
\end{equation*}
where the union is over all permutations $\sigma$ on $1,2,...,N$.  Note that $spt(\rho_N)$ is contained in $S$.  It is straightforward to see that $c(x_1,x_2,...,x_N) - \sum_{i=1}^Nu(x_i)$ is constant on $S$; we now show that any point where this function is minimized must belong to $S$, which will yield the desired result.

As $c =\infty$ whenever $x_i =x_j$ for any $i \neq j$, the minimum of this function must be attained at some point where $x_i \neq x_j$ for all $i \neq j$, ie, at some point where $c$ is smooth.  We must therefore have $\frac{\partial c}{\partial x_i}(x_1,x_2,...,x_N) -\frac{d u} {d x}(x_i)=0$ for all $i$.  We will show that this implies $(x_1,...,x_N) \in S$. 

First of all note that, for $(x_1,x_2,...,x_N) \in S$, a straightforward calculation confirms that $\frac{\partial c}{\partial x_j}(x_1,x_2,...,x_N) -\frac{d u} {d x}(x_j)=0$ for all $j$.

Now, choose a point  $(x_1,x_2,...,x_N)$ such that $\frac{\partial c}{\partial x_j}(x_1,x_2,...,x_N) -\frac{d u} {d x}(x_j)=0$ for all $j$ and assume, without loss of generality, that $x_1<x_2<...<x_N$.  By the pigeon hole principle, we have that $x_i \in I^i$ for at least one fixed $i$.  Now, note that for fixed $y_i =x_i$, the function 
\begin{equation*}
(y_1,...y_{i-1},y_{i+1},...,y_N) \mapsto c(y_1,...,y_N)-\sum_{j=1}^N u(y_j)
\end{equation*}
is strictly convex on the region 
\begin{equation*}
\{(y_1,...y_{i-1},y_{i+1},...,y_N):y_j<y_{j+1} \forall j=1,2,...N\} 
\end{equation*}
and so there is only one point  $(y_1,...y_{i-1},y_{i+1},...,y_n)$ in this region where   $\frac{\partial c}{\partial x_j}(x_1,x_2,...,x_N) -\frac{d u} {d x}(x_j)=0$.  We already have this equality when $y_j = (j-i)+x_i$ for all $j \neq i$ and so we must have $x_j = (j-i)+x_i$, which means $(x_1,x_2,...,x_N) \in S $, as desired.

\end{proof}

\newtheorem{nonunique}[local]{Remark}
\begin{nonunique}\label{nonunique}
(Non-uniqueness of the optimal measure)  It is interesting to note that the optimal measure in this case is not unique for $N>2$.  Indeed, note that $\rho_N$ is not symmetric; it's symmetrization $\tilde{\rho_N}$ is therefore another optimizer.  However, $\tilde{\rho_N}$ is the unique \emph{symmetric} optimizer.  In the next section, we will see that in higher dimensions, there can in fact be more than one \emph{symmetric} optimizer. 
\end{nonunique}

\subsection{An upper bound in higher dimensions}

We now construct our measure, by essentially coupling each line through the origin with the same line in each of the other copies of $\mathbb{R}^d$.  This coupling is done via the $1$-dimensional construction above.

It will be convenient to work in polar-like coordinates, as $\mathbb{R}^d$ is the (almost disjoint) union of lines through the origin.  Note that,  neglecting the origin, $\mathbb{R}^d = \mathbb{R} \times \mathbb{P}^{d-1}$, where $\mathbb{P}^{d-1}$ denotes $n-1$ dimensional projective space.  Let $r \in \mathbb{R}$ and $\theta \in \mathbb{P}^{d-1}$ represent the angular coordinates.  Note that fixing $\theta$ corresponds to fixing a line through the origin. Now, expressing the density $\rho(r,\theta)$ is these coordinates, we have, by Fubini's theorem,
\begin{equation*}
1=\int_{\mathbb{P}^{d-1}}\Big(\int_{\mathbb{R}}\rho(r,\theta)dr\Big) d\theta
\end{equation*}
and so $r\mapsto \rho(r,\theta)$ is integrable for almost all $\theta$.  For each fixed $\theta$, and $m=1,2,...,N$ let $r \mapsto F^m(r,\theta)$ be the optimal map obtained in the previous section, coupling the density $r \mapsto  \rho(r,\theta)$ to itself.  For $m=1,2,3,...N$, define $T^m(r,\theta) =(F^m(r,\theta), \theta)$. 

\newtheorem{measpres}{Proposition}[subsection]
\begin{measpres}
 $T^m$ pushes $\rho$ to itself.  
 \end{measpres}
 \begin{proof}
Note that $T^m$ is clearly bijective almost everywhere and it is smooth almost everywhere, as $F^m$ is smooth on the interior of each interval $I^i$.  The derivative of $T^m$, in $(r,\theta)$ coordinates, is (in block form):

\begin{equation*}
DT^m(r,\theta) = \begin{bmatrix}
D_rF^m(r,\theta) &  D_{\theta}F^m(r,\theta)\\
0 & I\\
\end{bmatrix}
\end{equation*}
The determinant of this upper triangular matrix is $D_rF^m(r,\theta) = \frac{\rho(r, \theta)}{\rho((F^m(r,\theta), \theta))} =\frac{\rho(r, \theta)}{(\rho(T^m(r,\theta))} $, as $F^m(r,\theta)$ pushes forward the density $\rho(r,\theta)$ to itself, for fixed $\theta$. By the change of variables formula, $(T^{m})_{\#}(\rho) = \rho$.

\end{proof}
Now set 
\begin{equation}\label{construct}
\rho_N= (T^1, T^2, T^3, ... ,T^N)_{\#} \rho.
\end{equation}
  By the preceding proposition, this yields a measure on $(\mathbb{R}^d)^N$ whose marginals are all $\rho$; that is, an element of $\Pi(\rho)$.  We immediately obtain
  
  \newtheorem{bound}[measpres]{Corollary}
  \begin{bound}
  For the $\rho_N$ defined in equation (\ref{construct}), we have:
  \begin{equation*}
  E_N(\rho) \leq C_N(\rho_N) =\int_{\mathbb{R}^{Nd}}\sum_{i\neq j}\frac{1}{|x_i-x_j|}d\rho_N(x_1,x_2,...,x_N)= \int_{\mathbb{R}^d}\sum_{i\neq j}\frac{1}{|T_i(x)-T_j(x)|}d\rho(x)
  \end{equation*}
  \end{bound}
  
\newtheorem{quality}[measpres]{Remark}
\begin{quality}
For measures which are not too spread out in the angular directions, we suspect our construction is nearly optimal.  Indeed, in the limit where the support of $\rho$ is a single line through the origin, the problem (\ref{mk}) reduces to the $1$-dimensional problem addressed in subsection 2.1.  On the other hand, our construction is far from optimal for measures which are not spread out in the radial direction.  Consider the limit when $\rho$ is concentrated at a single radius $r$.  In this case, the one dimensional "densities" $r \mapsto \rho(r,\theta)$ are each the sum of two Dirac masses; therefore, our construction yields an infinite total energy when $N >2$ and is thus far from optimal. 
\end{quality}
\subsection{Example: uniform radial density}
When $N=2$ and $\rho$ is radially symmetric, Cotar, Friesecke and Kl\"{u}ppelberg showed that the $\rho_N$ constructed above is in fact the optimizer in (\ref{mk}); that is, $E_2[\rho] =C_2(\rho_2)$ \cite{cfk}.  They also noted that, for general $N$, $E_2$ yields a lower bound on $E_N$: $E_N[\rho] \geq {N \choose 2}E_2[\rho]=\frac{N(N-1)}{2} E_2[\rho]$.  For radially symmetric measure, then, we now have both an upper and lower bound on $E_N$ which can be calculated explicitly.  Below, we compare these two bounds for a simple, radially symmetric $\rho$. 

\newtheorem{uniform}{Proposition}[subsection]
\begin{uniform}
Suppose that $\rho$ is the radial symmetric density given by $\rho(r) = \frac{1}{2}\chi_{[-1,1]}$.  Then $E_2[\rho]=1$ and $C_N[\rho_N] = \frac{-N^2}{2}+\frac{N}{2}+\frac{N^2}{2}(\sum_{n=1}^{N-1} \frac{1}{n})$, where $\rho_N$ is as in (\ref{construct}).
\end{uniform}

\newtheorem{logar}[uniform]{Remark}
\begin{logar}

Note that $\sum_{n=1}^{N-1} \frac{1}{n} \leq 1 +\ln(N-1)$.  Thus, we have 
\begin{equation*}
E_N[\rho] \leq C_N(\rho) \leq \frac{N}{2}+\frac{N^2ln(N-1)}{2}.
\end{equation*}
On the other hand, 
\begin{equation*}
E_N[\rho] \geq {N \choose 2}E_2[\rho] = \frac{N(N-1)}{2}.
\end{equation*}
The ratio of the upper to lower bound, then, is roughly logarithmic in $N$.

\end{logar}
\begin{proof}

For all $x$, we have one $F^m(x)$ in each interval $I^i$.  Note that, if $F^m(x) \in I^i$ and $F^n(x) \in I^j$, we have $|F^m(x) - F^n(x)| = \frac{2|i-j|}{N}$.  We then have:

\begin{equation*}
\sum_{m \neq n}\frac{1}{|F^m(x) - F^n(x)|} = \sum_{i\neq j}\frac{N}{2|i-j|}
\end{equation*}
It is straightforward to calculate that there are $N-1$ pairs $(i,j)$ where $|i-j| =1$, $N-2$ pairs where $|i-j| =2$, etc, ending in one pair where $|i-j| = N-1$.  Thus the above is equal to 

\begin{eqnarray*}
\sum_{i\neq j}\frac{N}{2|i-j|} &=& \frac{N}{2}(\frac{N-1}{1} + \frac{N-2}{2} + ....\frac{1}{N-1})\\
&=& \frac{N}{2}(N-1+\frac{N}{2}-1+\frac{N}{3}-1+...+\frac{N}{N-1} -1)\\
&=& \frac{N}{2}(N + \frac{N}{2} +\frac{N}{3} + ...+\frac{N}{N-1}-(N-1))\\
&=& \frac{N^2}{2}(1+\frac{1}{2} + \frac{1}{3} + ...+ \frac{1}{N-1}) - \frac{N}{2}(N-1)\\
&=& \frac{N^2}{2} \sum_{n=1}^{N-1} \frac{1}{n}+ \frac{N}{2} -\frac{N^2}{2}
\end{eqnarray*}
Integrating the constant function $\sum_{m \neq n}\frac{1}{|F^m(x) - F^n(x)|}$ against uniform measure on $[-1,1]$ now yields the desired result.
\end{proof}
\section{Reduction to a $1$-dimensional problem for radially symmetric measures}
 
We show in this section that, for radially symmetric measures, the problem can actually be reduced to an optimal transport problem with $1$-dimensional marginals. Set

\begin{equation}\label{redcoul}
h(r_1,r_2,...,r_N)= \inf_{|x_1|=r_1,|x_2|=r_2,...,|x_N| =r_N}c(x_1,x_2,...,x_N)\footnote{We take $c(x_1,x_2,...,x_N) = \sum_{i\neq j}\frac{1}{|x_i-x_j|}$ here, although a similar reduction will hold for other rotationally invariant costs.}
\end{equation}
Note that in this section, we will work in polar coordinates rather than the polar-like coordinates used earlier.  That is, we will identify $\mathbb{R}^d \approx [0,\infty) \times \mathbb{S}^{d-1}$.  Radial symmetry simply means that the density $\rho(r,\theta)=\rho(r)$ of the marginal $\rho=\rho(r,\theta)dr d\theta$ is independent of $\theta$.

For each fixed $\vec{r} =(r_1,r_2,...,r_N) \in \mathbb{R}^N$, fix 
\begin{equation*}
(x_1,x_2,...,x_N) = x_1(\vec{r}),x_2(\vec{r}),...,x_N(\vec{r})
\end{equation*}
attaining the infimum in \eqref{redcoul}; note that for any rotation $A \in SO(d)$, 
\begin{equation*}
(A(x_1),A(x_2),...A(x_N))
\end{equation*}
attains the infimum as well.  Take $S_{r_i} = \{ x\in \mathbb{R}^3: |x| = r_i\}$.  Now let $\nu$ be Haar measure on the rotation group $SO(d)$ and let $\gamma^{r_1,r_2,...,r_N} = G_{\#}\nu$, where 
\begin{equation*}
G:SO(3) \rightarrow S_{r_1} \times S_{r_2} \times ...\times S_{r_N}
\end{equation*}
is defined by $G(A) = (A(x_1),A(x_2),...,A(x_N))$.  Then $\gamma^{r_1,r_2,...,r_N}$ is a probability measure on $S_{r_1} \times S_{r_2} \times ...\times S_{r_N}$ with uniform marginals.

Then consider the optimal transportation problem with $1$-dimensional marginals $\mu := \rho(r)dr$ and cost function $h$; that is, minimize
\begin{equation}\label{reduced}
 \int_{\mathbb{R}^N} h(r_1,r_2,...,r_N)d\gamma
\end{equation}
among measures $\gamma$ on $\mathbb{R}^N$ whose $1$-dimensional marginals are all $\mu$.

Let $\tilde{\gamma}$ be an optimizer in (\ref{reduced}) (existence is guaranteed by standard results in optimal transport theory) and set 
\begin{eqnarray}\label{redsol}
\rho_N(x_1,x_2,...,x_N)&=& \rho_N(r_1,,r_2,...,r_N,\theta_1,\theta_2,...,\theta_N) \nonumber\\
&=:& \tilde{\gamma}(r_1,r_2,...,r_N) \otimes \gamma^{r_1,r_2,...,r_N}(\theta_1,\theta_2,...,\theta_N). 
\end{eqnarray}
\newtheorem{character}{Theorem}[subsection]
\begin{character}\label{character}
Assume that $\rho$ is radially symmetric.  Then the $\rho_N$ defined in \ref{redsol} is optimal for (\ref{mk}).
\end{character}
\begin{proof}
It is easy to see that $\rho_N$ has the appropriate marginals.  Now, the Kantorovich duality theorem \cite{K} and the optimality of $\tilde{\gamma}$ imply the existence of functions $u_1,...,u_N: (0, \infty) \rightarrow \mathbb{R}$ such that $h(r_1,...,r_N) - \sum_{i=1}^{N} u_i(r_i) \geq 0$ with equality when $(r_1, ...,r_N) \in spt (\tilde{\gamma})$.  We then have, for all $x_1,x_2,...,x_N$

\begin{eqnarray*}
c(x_1,...x_N)- \sum_{i=1}^{N} u_i(|x_i|)& \geq &h(|x_1|,...,|x_N|)-\sum_{i=1}^{N} u_i(|x_i|)\\
&\geq& 0
\end{eqnarray*}
with equality when $x_1,...,x_N$ attain the infimum in (\ref{redcoul}), \textit{and} $|x_1|,...|x_N| \in spt(\tilde{\gamma})$; that is, when $(x_1,...,x_N) \in spt(\rho_N)$.  Now, integrating with respect to any $\tilde{\rho_N} \in \Pi(\rho)$, we get
\begin{equation*}
\int_{\mathbb{R}^{Nd}} c(x_1,...x_N)d\tilde{\rho_N} \geq \sum_{i=1}^{N}\int_{\mathbb{R}^{Nd}} u_i(|x_i|)d\tilde{\rho_N}(x_1,x_2,...c_N) =  \sum_{i=1}^{N}\int_{\mathbb{R}^{d}} u_i(|x_i|)d\rho(x_i)
\end{equation*}
Now, noting that we have equality for $\tilde{\rho_N} = \rho_N$ and the right hand side is \textit{independent} of $\tilde{\rho_N} \in \Pi(\rho)$ yields the desired result.
\end{proof}
\subsection{The two marginal case}
As a consequence of Theorem \ref{character}, we obtain an easy proof of the following result, already implicit in \cite{seidl} (without rigorous justification) and also present in \cite{cfk} and \cite{bdpgg} (with more complicated proofs). 
\newtheorem{2marg}[character]{Corollary}

\begin{2marg}
When $N=2$ and $\rho_1 =\rho_1(r)$ is radially symmetric, the solution is unique and is concentrated on the graph of the function $x \mapsto -\frac{x}{|x|}f(|x|)$, where $f: (0, \infty) \mapsto (0, \infty)$ is defined implicitly by
\begin{equation}\label{funct}
\int_{r}^{-\infty} \rho_1(s)ds =\int_{0}^{f(r)} \rho_1(s)ds
\end{equation}
 
\end{2marg}

\begin{proof}
It is easy to see that the minimum in (\ref{redcoul}) is uniquely attained at antipodal points $x_2 =-\frac{r_2}{r_1}x_1$ and so $h(r_1,r_2)= \frac{1}{r_1+r_2}$.  Noting that $\frac{\partial^2 h}{\partial r_i \partial r_j} > 0$, a classical result in optimal transportation implies that the unique minimizer for the reduced problem (\ref{reduced}) is concentrated on the graph of a decreasing function, $r_2 =f(r_1)$.  As there is exactly one decreasing function pushing the measure $\mu$ to itself (namely \eqref{funct}), this implies the desired result.
\end{proof}
\subsection{The multi-marginal case}
Finally, we use Proposition \ref{character} to assert two new qualitative facts about the $N \geq 3$ case.  First, we show that the solution may be concentrated on a set with dimension \textit{greater} than $d$ (note that Monge type solutions, and their symmetrizations, would have dimension $d$).  Secondly, the solution may be \emph{non-unique}, even with an additional symmetry condition imposed; that is, there can be more than one minimizer in \eqref{mk} which is symmetric in the arguments $(x_1,x_2,...,x_N)$.

Our results in these directions (Propositions \ref{highd} and \ref{nonuniquesymm}) rely on certain assumptions on the structure on the support of the minimizers.  We suspect that these conditions hold generally, at least when the number of electrons is large.  Below (Lemma \ref{spans}) we provide an explicit example of a marginal $\rho$ for which optimizer satisfies the conditions in Proposition \ref{highd}.  Numerical calculations in \cite{sggs} suggest that the conditions in both Propositions \ref{highd} and \ref{nonuniquesymm} hold for other marginals (see Remark \ref{atom} below).

\newtheorem{highd}[character]{Proposition}

\begin{highd}\label{highd}
Suppose $N \geq 3$ and $d\geq 3$.   Let $\rho$ be a radially symmetric measure, and assume that for some optimizer $\gamma$ in \eqref{reduced} the following condition holds.

For all $r$ in some subset $I \subseteq (0,\infty)$ of positive measure, there is at least one point $\vec{r}=(r_1,r_2,r_3,...r_N) \in spt(\gamma)$, with $r_1=r$ for which there exist minimizing $(x_1(\vec{r}),x_2(\vec{r}),...,x_N(\vec{r}))$ in \eqref{redcoul} such that $x_1(\vec{r})$ and $x_i(\vec{r})$ are not co-linear, for some $i$.   
 
Then there exists an optimizer $\rho_N$ in \eqref{mk} whose support is at least $(2d-2)$-dimensional.
\end{highd}
\begin{proof}
Choose $(x_1(\vec{r}),x_2(\vec{r}),...,x_N(\vec{r}))$ as in the hypothesis; without loss of generality, assume $x_1(\vec{r})$ and $x_2(\vec{r})$ are not co-linear.  Now, it is easy to check that the set 
\begin{equation*}
\{A(x_2(\vec{r})): A \in SO(d), \text{ }A(x_1(\vec{r}))=x_1(\vec{r})\} = \{y \in \mathbb{R}^d: y \cdot x_1(\vec{r}) = x_2(\vec{r}) \cdot x_1(\vec{r}) \text{ and } |y| =|x_2(\vec{r})| \}
\end{equation*}
is at least $(d-2)$-dimensional. Therefore, the set 
\begin{eqnarray*}
\{(A(x_1(\vec{r})),A(x_2(\vec{r})),....,A(x_N(\vec{r}))): A \in SO(d), \text{ }A(x_1(\vec{r}))=x_1(\vec{r})\}\\
=\{(x_1(\vec{r}),A(x_2(\vec{r})),....,A(x_N(\vec{r}))): A \in SO(d), \text{ }A(x_1(\vec{r}))=x_1(\vec{r})\} 
\end{eqnarray*}
is $(d-2)$-dimensional.

To summarize, this implies that for each $x_1 \in \mathbb{R}^d$ with radius $|x_1| =r \in I$, there is a $(d-2)$-dimensional set of points 
\begin{equation*}
\{(x_1,A(x_2),....,A(x_N)):A \in SO(d), A(x_1)=x_1)\}
\end{equation*}
in the support of the optimizer; the result follows. 
\end{proof}
We suspect that the condition on the non co-linearity of the optimally coupled vectors $x_1,x_2,...,x_N$ is generically true.  Below, we prove that this condition holds for a specific example marginal $\rho$.

\newtheorem{spans}[character]{Lemma}
\begin{spans}\label{spans}
Let  $N \geq 3$ and fix $r_1,r_2,....,r_N \in [1-\epsilon, 1+\epsilon]$.  Suppose $x_1,x_2,... ,x_N \in argmin_{|x_i|=r_i} c(x_1,x_2,...,x_N)$.  Then, for $\epsilon$ sufficiently small, and any distinct $i,j,k$, $x_i, x_j,x_k$ are not co-linear.  
\end{spans}
\begin{proof}
If $x_i,x_j,x_k$ were co-linear, at least two of them, say $x_i$ and $x_j$, would have to point in the same direction, which case $|x_i-x_j| =|r_i-r_j|<2\epsilon$.  But then $c(x_1,x_2,...,x_N) > \frac{1}{|x_i-x_j|} > \frac{1}{2\epsilon}$, which is a contradiction for small $\epsilon$ as other configurations clearly have lower total cost.
\end{proof}

\newtheorem{highdexplicit}[character]{Corollary}
\begin{highdexplicit}
Assume $N \geq 3$ and $d \geq 3$.  Then there exists a measure $\rho$ on $\mathbb{R}^d$ for which there is a measure $\rho_N$, on $\mathbb{R}^{Nd}$, optimal in \eqref{mk}, whose support is at least $(2d-2)$-dimensional. 
\end{highdexplicit}
\begin{proof}
Choose an absolutely continuous, radially symmetric marginal $\rho$, so that the corresponding measure $\mu:=\rho(r)dr$ is concentrated around $r=1$; precisely,

\begin{equation*}
\mu(\{r\in[1-\epsilon, 1+\epsilon]\}) >1-\frac{1}{2N},
\end{equation*}
with $\epsilon$ as in the previous lemma.  Then for any measure $\gamma \in \Pi(\mu)$, and any $i=1,2,...N$, we have 
\begin{equation*}
\gamma \big(\{(r_1,r_2,...,r_N): r_i \notin [1-\epsilon, 1+\epsilon]\}\big) = \mu ([1-\epsilon, 1+\epsilon]^C) < \frac{1}{2N}.
\end{equation*}
Therefore, 
  \begin{equation*}
  \gamma\Big(\bigcup_{i=1}^N \{(r_1,r_2,...,r_N): r_i \notin [1-\epsilon, 1+\epsilon]\}\Big) < \frac{1}{2}.
  \end{equation*}
Noting that the complement of $\bigcup_{i=1}^N \{(r_1,r_2,...,r_N): r_i \notin [1-\epsilon, 1+\epsilon]\}$ is $[1-\epsilon, 1+\epsilon]^N$, this implies $\gamma ([1-\epsilon, 1+\epsilon]^N) >\frac{1}{2}$.  

It follows that, for the optimal measure $\gamma$ in \eqref{redcoul}, letting $I$ be the projection of  $[1-\epsilon,1+\epsilon]^N \cap spt(\gamma)$, we have $\mu(I)>0$ and by absolute continuity, $I$ has positive Lebesgue measure.  By the preceding lemma, the condition in Proposition \ref{highd} is satisfied, and therefore, the result follows.
\end{proof}

\newtheorem{nonuniquesymm}[character]{Proposition}

\begin{nonuniquesymm}\label{nonuniquesymm}
Suppose $N \geq 3$ and $d \geq 2$.  Let the marginal $\rho$ be radially symmetric, and assume that for some optimizer in \eqref{reduced} there is \emph{at least one} point $\vec{r}=(r_1,r_2,r_3,...,r_N) \in spt(\gamma)$ for which the following two conditions hold:

\begin{enumerate}
\item The radii $r_1,r_2,...,r_N$ are distinct.
\item There exist $x_1(\vec{r}),x_2(\vec{r}),...,x_N(\vec{r})$ minimizing \eqref{redcoul} that span $\mathbb{R}^d$.
\end{enumerate}

Then there is more than one symmetric minimizer $\rho_N$ in \ref{mk}. 
\end{nonuniquesymm}
Note that the first condition seems to be generic; it would be surprising if the optimizer coupled exclusively points with common radii together.  The second condition cannot hold if the dimension $d$ is larger than the number of electrons $N$.  In fact,  it is straightforward, using Lagrange multipliers, to show that the minimizing vectors $x_1(\vec{r}),x_2(\vec{r}),...,x_N(\vec{r})$ are linearly dependent, and so the first condition cannot hold when $d =N$ either.  Nonetheless, I suspect the second condition holds generically when $d<N$. 
\begin{proof}
Note that can choose the minimizers $(x_1(\vec{r}),x_2(\vec{r}),...,x_N(\vec{r}))$ in \eqref{redcoul} to be symmetric whenever the radii are distinct; that is, for any permutation $\sigma$ on $1,2,...N$, we can assume 
\begin{equation*}
x_i(r_1,r_2,...r_N) = x_{\sigma(i)}(r_{\sigma(1)},r_{\sigma(2)}...r_{\sigma(N)}).
\end{equation*}

Now, the procedure above produces an optimal measure $\rho_N$, defined by \eqref{redsol};  note that we can find another minimizer $\overline{\rho_N}$ simply by replacing in the procedure leading up to \eqref{redsol} Haar measure $\nu$ on $SO(d)$ with Haar measure on the whole orthogonal group $O(d)$. Let $\tilde{\rho}$ and $\tilde{\overline{\rho}}$ be the symmetrizations of $\rho$ and $\overline{\rho}$, respectively.  We need only to verify that $\tilde{\rho_N} \neq \tilde{\overline{\rho_N}}$.  To do this, we will exhibit a point in the support of $\tilde{\overline{\rho_N}}$ which is not in the support of  $\tilde{\rho_N}$.

Choose $\vec{r}=(r_1,r_2,...,r_N) \in spt(\gamma)$ with distinct radii and $(x_1,x_2,...,x_N):=(x_1(\vec{r}),x_2(\vec{r}),...,x_N(\vec{r})) \in spt(\rho_N)$ which span $\mathbb{R}^d$.  It follows that for \textit{any} $\overline{A} \in O(d)$ such that $\overline{A} \notin SO(d)$, we have
\begin{equation*}
(\overline{A}(x_1),\overline{A}(x_2),...,\overline{A}(x_N)) \neq (A(x_1),A(x_2),...,A(x_N))
\end{equation*}
for \textit{all} $A \in SO(3)$ (as if $(\overline{A}(x_1),\overline{A}(x_2),...,\overline{A}(x_N)) = (A(x_1),A(x_2),...,A(x_N))$, then $A = \overline{A}$ by the span condition.  It now follows that 
\begin{equation*}
(\overline{A}(x_1),\overline{A}(x_2),...,\overline{A}(x_N)) \in spt(\overline{\rho_N}) \subseteq spt(\tilde{\overline{\rho_N}})
\end{equation*}
Now, it follows from the construction that the only points $z_1,z_2,...,z_N$ in the support of $\rho_N$ with $|z_i| = r_i$ for all $i=1,2,..,N$ are points of the form $(A(x_1),A(x_2),...,A(x_N))$ with $A \in SO(d)$.  The only points $z_1,z_2,...,z_N$ in the support of $\tilde{\rho_N}$ with $|z_i| = r_i$ are points of the form $(A(x_{\sigma(1)}(\vec{r}_{\sigma})),A(x_{\sigma(2)}(\vec{r}_{\sigma})),...,A(x_{\sigma(N)}(\vec{r}_{\sigma})))$ with $A \in SO(d)$, where $\sigma \in S_N$ is a permutation and 

\begin{equation*}
\vec{r}_{\sigma} = (r_{\sigma(1)},r_{\sigma(2)},...,r_{\sigma(N)})
\end{equation*}
But as the minimizers $(x_1(\vec{r}),x_2(\vec{r})...,x_N(\vec{r}))$ are symmetric, $(A(x_{\sigma(1)}(\vec{r}_{\sigma})),A(x_{\sigma(2)}(\vec{r}_{\sigma})),...,A(x_{\sigma(N)}(\vec{r}_{\sigma})))=(A(x_1),A(x_2),...,A(x_N))$, and so
\begin{equation*}
(\overline{A}(x_1),\overline{A}(x_2),...,\overline{A}(x_N)) \notin spt(\rho_N),
\end{equation*}
from which the result follows.

\end{proof}

\newtheorem{atom}[character]{Remark}
\begin{atom}\label{atom}\textbf{(Numerical evidence in support of the hypotheses)}
In \cite{sggs}, the authors compute what they call co-motion functions.  Using spherically symmetric densities $\rho$ on $\mathbb{R}^3$ corresponding to the ground states of the Lithium atom (N=3) and Beryllium atom (N=4), they find functions $f_i$, $i=2,...,N$, such that, in our notation, $\gamma:=(Id,f_2,...,f_N)|_{\#} \mu \in \Pi(\mu)$.  They also numerically construct potential functions $v(x)$ and verify numerically that 

\begin{equation*}
(x_1,x_2,...,x_N) \mapsto \sum_{i\neq j}^N\frac{1}{|x_i-x_j|} -\sum_{i=1}^N v(x_i)
\end{equation*}
is minimized at each point on the set $\{x,f_2(x),...,f_N(x)\}$.  This easily implies that, in our notation, $\gamma$ is optimal in \eqref{reduced}. They also compute the minimizing angles numerically, and one can readily verify from their calculations that:

\begin{enumerate}
\item For all $x$, the radii in $|x|, |f_2(x)|,...,|f_N(x)|$ are all distinct.
\item For the Beryllium density with $N=4$, the vectors $x,f_2(x),...f_N(x)$ are linearly independent. 
\end{enumerate}
 
Proposition \ref{highd} then implies that for these densities there are optimal measures $\rho_N$ with at least $2d-2=4$-dimensional support, while Proposition \ref{nonuniquesymm} implies that there is more than one symmetric optimal measure for the Beryllium density with $N=4$.
\end{atom}
\bibliographystyle{plain}
\bibliography{biblio}
\end{document}